\documentclass[a4paper,11pt]{amsart}

\usepackage{verbatim}
\usepackage{amssymb}
\usepackage{amscd}
\usepackage{amsfonts,esint}
\usepackage{indentfirst}
\usepackage{verbatim}
\usepackage{amsmath}
\usepackage{amsthm}
\usepackage{enumerate}
\usepackage{graphicx}
\usepackage{subfig}
\usepackage{color}
\usepackage[OT1]{fontenc}
\usepackage[latin1]{inputenc}
\usepackage[english]{babel}
\usepackage{amssymb}
\usepackage{cancel}

\DeclareMathOperator{\sign}{sign}

\usepackage[active]{srcltx}
\usepackage[colorlinks,pdfpagelabels,pdfstartview = FitH,bookmarksopen = true,bookmarksnumbered = true,linkcolor = blue,plainpages = false,hypertexnames = false,citecolor = red] {hyperref}

\allowdisplaybreaks
\title[]{ \large A simple proof of the generalized Leibniz rule \\  on bounded Euclidean domains}

\author{Quoc-Hung Nguyen}
\author{Yannick Sire}
\author{Juan-Luis V\'azquez}

\address{ShanghaiTech University,
	393 Middle Huaxia Road, Pudong,
	Shanghai, 201210, China}
\email{qhnguyen@shanghaitech.edu.cn }
\address{Department of Mathematics, Johns Hopkins University, Baltimore, MD 21218, United States}
\email{ysire1@jhu.edu}
\address{Departamento de Matem\'aticas, Universidad Aut\'onoma de Madrid, 28049 Madrid, Spain}
\email{ juanluis.vazquez@uam.es}

\newtheorem{theorem}{Theorem}[section]

\theoremstyle{definition}

\newtheorem{remark}{Remark}[section]

\numberwithin{equation}{section}

\def\eqn#1$$#2$${\begin{equation}\label#1#2\end{equation}}
\def\charfn_#1{{\raise1.2pt\hbox{$\chi_{\kern-1pt\lower3pt\hbox{{$\scriptstyle#1$}}}$}}}

\def\dist{\operatorname{dist}}

\usepackage{dsfont}

\def\er{\mathbb R}

\def\mean#1{\mathchoice%
          {\mathop{\kern 0.2em\vrule width 0.6em height 0.69678ex depth -0.58065ex
                  \kern -0.8em \intop}\nolimits_{\kern -0.4em#1}}%
          {\mathop{\kern 0.1em\vrule width 0.5em height 0.69678ex depth -0.60387ex
                  \kern -0.6em \intop}\nolimits_{#1}}%
          {\mathop{\kern 0.1em\vrule width 0.5em height 0.69678ex
              depth -0.60387ex
                  \kern -0.6em \intop}\nolimits_{#1}}%
          {\mathop{\kern 0.1em\vrule width 0.5em height 0.69678ex depth -0.60387ex
                  \kern -0.6em \intop}\nolimits_{#1}}}

\def\vintslides_#1{\mathchoice%
          {\mathop{\kern 0.1em\vrule width 0.5em height 0.697ex depth -0.581ex
                  \kern -0.6em \intop}\nolimits_{\kern -0.4em#1}}%
          {\mathop{\kern 0.1em\vrule width 0.3em height 0.697ex depth -0.604ex
                  \kern -0.4em \intop}\nolimits_{#1}}%
          {\mathop{\kern 0.1em\vrule width 0.3em height 0.697ex depth -0.604ex
                  \kern -0.4em \intop}\nolimits_{#1}}%
          {\mathop{\kern 0.1em\vrule width 0.3em height 0.697ex depth -0.604ex
                  \kern -0.4em \intop}\nolimits_{#1}}}

\newcommand{\aveint}[2]{\mathchoice%
          {\mathop{\kern 0.2em\vrule width 0.6em height 0.69678ex depth -0.58065ex
                  \kern -0.8em \intop}\nolimits_{\kern -0.45em#1}^{#2}}%
          {\mathop{\kern 0.1em\vrule width 0.5em height 0.69678ex depth -0.60387ex
                  \kern -0.6em \intop}\nolimits_{#1}^{#2}}%
          {\mathop{\kern 0.1em\vrule width 0.5em height 0.69678ex depth -0.60387ex
                  \kern -0.6em \intop}\nolimits_{#1}^{#2}}%
          {\mathop{\kern 0.1em\vrule width 0.5em height 0.69678ex depth -0.60387ex
                  \kern -0.6em \intop}\nolimits_{#1}^{#2}}}

\newtoks\by
\newtoks\paper
\newtoks\book
\newtoks\jour
\newtoks\yr
\newtoks\pages
\newtoks\vol
\newtoks\publ

\def\name[#1, #2]{#1 #2}
\def\ota{{\hbox{\bf ???}}}
\def\cLear{\by=\ota\paper=\ota\book=\ota\jour=\ota\yr=\ota
\pages=\ota\vol=\ota\publ=\ota}
\def\endpaper{\the\by, \textit{\the\paper},
{\the\jour} \textbf{\the\vol} (\the\yr), \the\pages.\cLear}
\def\endbook{\the\by, \textit{\the\book},
\the\publ, \the\yr.\cLear}
\def\endpap{\the\by, \textit{\the\paper}, \the\jour.\cLear}
\def\endproc{\the\by, \textit{\the\paper}, \the\book, \the\publ,
\the\yr, \the\pages.\cLear}

\definecolor{darkgreen}{rgb}{0.1, .65, .1}

{
\newcommand{\nc}{\normalcolor}

\begin{document}

\begin{abstract}
This note is devoted to a simple proof of the generalized Leibniz rule  in bounded domains. The operators under consideration are the so-called spectral Laplacian and the restricted Laplacian. Equations involving such operators have been lately considered by Constantin and Ignatova in the framework of the SQG equation \cite{CI} in bounded domains, and by  two of the  authors \cite{HungJuan} in the framework of the porous medium with nonlocal pressure in bounded domains. We will use the estimates in this work in a forthcoming paper on the study of Porous Medium Equations with pressure given by Riesz-type potentials.
\end{abstract}

\maketitle

\tableofcontents

\section{Introduction}	

Commutator estimates are instrumental in the study of PDEs. Recently, several models arising in Fluid dynamics for instance, involve Fourier multipliers in $\mathbb R^n$ of the type $|\xi|^{2\alpha}$ where $\alpha >0$ is real.  The case of bounded domains, hence outside of the theory of multipliers, has been considered  e.g. in \cite{CI}. \nc

The purpose of the present note is to provide a very simple proof of a commutator estimate in the setting of operators defined on bounded domains. In such a case, one cannot rely on the classical theory of para-differential operators, and one has to use other types of tools to derive the estimates. Heat kernel methods are very  useful in this setting since they require very little a priori assumptions on the domain. Here we consider another approach, much more general somehow, relying on harmonic extensions as in  \cite{caffaSilvestre,Stinga},  and inspired by \cite[Lemma 2]{HungJuan}.

\vskip-1cm

\subsection*{Fractional operators on bounded domains}

In the following we consider $\Omega$ a bounded smooth domain of $\mathbb R^d$ with $d \geq 1$. We first list the two nonlocal operators that serve as motivation.

\subsubsection*{The Spectral Laplacian}
Consider the eigensystem ($\varphi_j,\lambda_j)$ for the Dirichlet Laplacian acting in $\Omega$, namely
$-\Delta \varphi_j=\lambda_j\varphi_j$ with homogeneous Dirichlet boundary conditions. It is well known that $0<\lambda_1\leq...\leq\lambda_j\leq...$, that $\lambda_j \asymp j^{2/d}$, and that $-\Delta$ is a positive self-adjoint operator in $L^2(\Omega)$ with
domain $D(-\Delta) = H^2(\Omega)\cap H^1_0(\Omega)$. The ground state $\varphi_1$ is positive and $\varphi_1(x)\asymp d(x)$ for all $x\in \Omega$, where $d(x)$ denotes distance to boundary. For all $0<\alpha<1$  we define the  \sl spectral fractional Laplacian\rm,  $(-\Delta)_{sp}^{\alpha}$, by
\begin{equation}
(-\Delta)_{sp}^{\alpha}u=\sum_{j=1}^{\infty}\lambda_j^{\alpha}u_j \varphi_j,
\qquad u_j=\int_{\Omega}u(x)\varphi_j(x)dx\,.
\end{equation}

\subsubsection*{The Restricted Laplacian}

One can define a fractional Laplacian operator by using the integral representation in terms of hypersingular kernels (for instance for locally $C^{1,1}$ functions and bounded over $\mathbb R^d$)
\begin{equation}\label{sLapl.Rd.Kernel}
(-\Delta_{\mathbb R^d})^{\alpha}  g(x)= c_{d,s}\mbox{
P.V.}\int_{\mathbb{R}^d} \frac{g(x)-g(z)}{|x-z|^{d+2\alpha}}\,dz,
\end{equation}
where $c_{d,\alpha}>0$ is a normalization constant. In this case we materialize the zero Dirichlet condition  by restricting the operator to act only on functions that are zero outside $\Omega$. We will call  the operator defined in such a way the \textit{restricted fractional Laplacian}.

It is a well-known fact that the two previously defined operators are {\sl different}.  This is easily recognized by simple properties, like the fact that they have different spectral sequences, or the fact that the solutions of the Dirichlet Problem have different boundary behaviour and a different Green function, see \cite{BonSirVaz, Bucur, CarrCIME, ChenSong, MusNaz14, ServVa14}.

\medskip

\subsection*{{ Commutator results and the idea of the proof}}

Our main results are the following. The first deals with our main contribution, i.\,e., a Leibniz rule for the operator defined above as the spectral Laplacian.

\begin{theorem} \label{thm-kapon1}  Let $0<\beta\leq \alpha< { \frac{1}{2}}$ and  $f,g,h \in C^\infty_0(\Omega)$. Let $(-\Delta)_{sp}^\alpha$ denote the {\sl spectral Laplacian} acting in a smooth bounded domain $\Omega$ in $\mathbb R^d$, $d\geq 1$. There exists $C>0$ such that
	\begin{align}\label{kato-ponce1}
	&||(-\Delta)_{sp}^{\alpha}(gh)-g\,(-\Delta)_{sp}^{\alpha} h- h \, (-\Delta)_{sp}^{\alpha}g||_{L^2(\Omega)}\\
	&\leq C ||(-\Delta)_{sp}^{\frac{\beta}{2}} h||_{L^2(\Omega)}  ||g||_{L^\infty(\Omega)}^{\frac{\beta}{2\alpha}}||(-\Delta)_{sp}^\alpha g||_{L^\infty(\Omega)}^{\frac{2\alpha-\beta}{2\alpha}}.\nonumber
		\end{align}
		
\end{theorem}

 As a matter of fact, the proof of Theorem \ref{thm-kapon1} is an adaptation of a strategy that can be easily implemented also in the case of the fractional Laplacian defined  {\sl in the whole of $\mathbb R^d$}. We state then the following

\begin{theorem} \label{thm-kapon2}  Let $0<\beta\leq \alpha<{\frac{1}{2}}$ and  $f,g,h \in C^\infty_0(\mathbb R^d)$ for $d \geq 1$. Let $(-\Delta)^\alpha$  denote the operator with Fourier multiplier  $|\xi|^{2\alpha}$ in $\mathbb R^d$ with $d\geq 1$. There exists $C>0$ such that
	\begin{align}\label{kato-ponce2}
	&||(-\Delta)^{\alpha}(gh)-g\,(-\Delta)^{\alpha} h- h \, (-\Delta)^{\alpha}g||_{L^2(\mathbb R^d)}\\
	&\leq C ||(-\Delta)^{\frac{\beta}{2}} h||_{L^2(\mathbb R^d)}  ||g||_{L^\infty(\mathbb R^d)}^{\frac{\beta}{2\alpha}}||(-\Delta)^\alpha g||_{L^\infty(\mathbb R^d)}^{\frac{2\alpha-\beta}{2\alpha}}.\nonumber
	\end{align}
\end{theorem}

The previous theorems whenever $(-\Delta)^\alpha$ is  the standard fractional Laplacian with Fourier symbol  \ $|\xi|^{2\alpha}$  are  known as a version of a commutator estimate by Kenig, Ponce and Vega \cite{KPV} (see also the famous Kato-Ponce estimate \cite{KP}). {More precisely, the Kenig-Ponce-Vega estimate in $\mathbb R^d$  writes}
$$
||(-\Delta)^{\alpha}(gh)-g\,(-\Delta)^{\alpha} h- h \, (-\Delta)^{\alpha}g||_{L^p} \lesssim \|(-\Delta)^{\alpha_1} g\|_{L^{p_1}}\|(-\Delta)^{\alpha_2} h\|_{L^{p_2}}
$$
whenever $\alpha=\alpha_1+\alpha_2$, $0 < \alpha, \alpha_1,\alpha_2 <1/2$, $\frac1p=\frac{1}{p_1}+\frac{1}{p_2}$, $1<p,p_1,p_2<\infty$.

We would like also to point out that recently in \cite{dong} D. Li obtained
for the first time a general Leibniz rule for all fractional Laplacian
operators $(-\Delta)^\alpha$ including the difficult end-point cases,
solving in particular a conjecture by Kato and Ponce. Since we are interested for our applications in the $L^2$ estimate, and since the proof is very simple and straightforward in this case, we focus only on these particular exponents. It must be said that  the full estimates for the left-hand side in our main theorem in other Lebesgue spaces, which is important for other applications, can be obtained following the same type of strategy. However, in this latter case, instead of using simple integration by parts followed by H\"older and/or Hardy inequality, one needs to invoke Cald\'eron-Zygmund inequalities and interpolation theory for instance. {A full account on such estimates in the Euclidean space using Littlewood-Paley theory can be found in the mentioned paper by Li \cite{dong}. }

\begin{remark}
It will be clear from the proof that the case $\beta=0$ is allowed. This is actually a version of Theorem 1.2 in \cite{dong}. In this later, the author managed even, invoking Coifman-Meyer theorem for instance, to replace the $L^\infty$ norm by the weaker $BMO$ norm.
\end{remark}

The proof of Theorems \ref{thm-kapon1} and \ref{thm-kapon2} is based on the following idea: owing to \cite{caffaSilvestre,Stinga}, powers of the Laplacian can be realized as the Dirichlet-to-Neumann operator of a suitable extension. Such extensions are also valid for the spectral and restricted Laplacian, as described more precisely below.   This allows to use the elliptic PDE satisfied in the extended domain via integration by parts to get the desired cancellation. A unified theory for several non-trivial sharp commutator estimates based on this idea has also  been extensively studied by Lenzmann and Schikorra \cite{LenzSchi} (see also \cite{BSS} for a geometric version of such extension). Our proof is inspired also by the one of \cite[Lemma 2]{HungJuan}. {It should be noticed that one can find higher order Kato-Ponce and Leibniz commutator estimates in \cite{dong}. Since we make use of a PDE technique via an extension, the case of higher order operators $\alpha \geq 1/2$ needs to be treated differently. Even for $\alpha \geq 1$, it requires the extension formulae developed in \cite{changYang} and a rather technical adaptation of the present simple argument. This is under investigation by Dong Li and the second author \cite{dongYannick}. 

We would like to emphasize that in Theorem \ref{thm-kapon2} the norms are taken in the whole space $\mathbb R^d$. If one considers the {\sl restricted Laplacian}, one could ask if one can take the norms in the domain $\Omega$ where the functions $g$ and $h$ are supported: the analogue of the estimate \eqref{kato-ponce1} whenever norms are taken in $\Omega$ is an open problem and we conjecture that in general such an estimate does not hold. The reason why the latter seems to be unlikely to hold is  probably due to a combination of the tail effects of the restricted Laplacian and the boundary behaviour of the functions  (which relies crucially on the parameter $\alpha$). However, we conjecture that there should be a natural replacement for the Leibniz rule in this case. In the appendix we provide two estimates supporting our conjecture.

\medskip

We recall that the fractional Laplacian in Theorem  \ref{CaffaS} whenever the functions are supported in $\Omega \subset \mathbb R^d$ is not the so-called {\sl regional fractional Laplacian} which is defined by
\begin{equation}
(-\Delta)_{\Omega}^\alpha g:= \int_{\Omega} \frac{g(x)-g(y)}{|x-y|^{d+2\alpha}} dy ~~\text{in}~~\Omega.
\end{equation}

These three {\sl different} operators play an important role in the theory of L\'evy processes (see e.g. the recent papers \cite{KSV1,KSV2}). However, we are not aware of any kind of replacement of Theorem \ref{CaffaS} in the case of the censored Laplacian, and as a consequence we cannot run the same argument to prove a generalized Leibniz rule. Furthermore, the censored Laplacian is not a pseudo-differential operator in the classical sense so that standard techniques do not seem to apply successfully.

\subsection*{ Review of the extension results} We now state the two versions of the extension problem we need to implement the strategy of the proof of our main result.
We state first the main result in \cite{Stinga}, Theorem 1.1 (see also\cite{cabreTan,CDDS}), which allows us to deal with the Spectral Laplacian.

\begin{theorem}\label{ST}
Let $d \geq 1$ and $\alpha \in (0,1)$. Consider $g \in C^{\infty}_0(\Omega)$; then the function defined by
\begin{equation}\label{formulaBVP}
	u(x,y)=\frac{y^{2\alpha}}{4^\alpha\Gamma(\alpha)}\int_{0}^{\infty}e^{-\frac{y^2}{4t}}e^{t\Delta_{sp}}g(x)
\frac{dt}{t^{1+\alpha}}=\frac{1}{\Gamma(\alpha)}\int_{0}^{\infty}e^{-\frac{y^2}{4t}}e^{t\Delta_{sp}}
\left[(-\Delta)_{sp}^\alpha g\right](x)\frac{dt}{t^{1-\alpha}}
	\end{equation}

	solves in the weak sense the boundary value problem
	\begin{equation*}
\left\{
\begin{array}
{ll}%
\operatorname{div}(y^{1-2\alpha}\nabla u)=0~~&\text{in }\Omega\times(0,\infty),\\
u=0 &\text{on}
~~\partial\Omega\times (0,\infty),
\\
u(x,0)=g(x) &\text{in}
~~\Omega.
\\
\end{array}
\right.
\end{equation*}
Furthermore, one has pointwise
\begin{align*}
-\lim\limits_{y\to 0^+}y^{1-2\alpha} u_y(x,y)= (-\Delta)_{sp}^\alpha g(x).
\end{align*}
\end{theorem}

\begin{remark}
Theorem \ref{ST} holds in a much more general setting than the one used here. In particular, it opens the way to get a wide class of commutator estimates in non-Euclidean frameworks (see \cite{BSS,BSS2}).
\end{remark}

Secondly, we state the one useful for the Fractional Laplacian, due to \cite{caffaSilvestre}.

\begin{theorem}\label{CaffaS}
Let $d \geq 1, \alpha \in (0,1)$ and $g \in C^{\infty}_0(\er^d)$. Denote by  $P_\alpha$ is the Poisson kernel of $\operatorname{div}(y^{1-2\alpha}\nabla \, \cdot)$ in $\mathbb R^{d+1}_+$, i.e.
\begin{equation*}
P_\alpha(x,y)=\frac{y^{2\alpha}}{(|x|^2+y^2)^{\frac{d+2\alpha}{2}}}.
\end{equation*}
 The function defined by
\begin{equation}\label{formulaBVP2}
	u(x,y)=\int_{\mathbb R^d}P_\alpha(x-t,y)g(t)\,dt
	\end{equation}
	solves in the weak sense the boundary value problem
	\begin{equation*}
\left\{
\begin{array}
{lll}%
\operatorname{div}(y^{1-2\alpha}\nabla u)=0~~&\text{in }\mathbb R_+^{d+1} ,\\
u(x,0)=g(x) &\text{on}
~~\partial \mathbb R^{d+1}_+ \sim \mathbb R^d.
\\
\end{array}
\right.
\end{equation*}
Furthermore, one has pointwise
\begin{align*}
-\lim\limits_{y\to 0^+}y^{1-2\alpha} u_y(x,y)= (-\Delta)^\alpha g(x).
\end{align*}
\end{theorem}

\begin{remark}
As noticed in \cite{Stinga}, by the very construction of $u$ in the previous theorems, the solution itself enjoys decay properties at $\infty$. In particular, this justifies the computations in the following sections.
\end{remark}

\medskip

\noindent { \bf Comments.} A classical technique to prove commutator estimates is to use Littlewood-Paley decomposition and Coifman-Meyer estimates combined with other deep harmonic analysis tools. We adopt here a different strategy based on the unified approach by the first two authors and then Lenzmann and Schikorra previously mentioned; that approach does not require at all in our case any deep result on boundedness of multipliers. It allows in particular to get {\sl sharp} estimates in much more general frameworks, at least for operators given by powers of the second order operators, than the paraproduct technique. Indeed, a major drawback of using paraproducts is their intrinsic Euclidean nature. However, Bernicot and Frey (see e.g. \cite{frey} and references therein) developed in the last decade a theory of para-differential calculus on geometric settings where the standard Fourier analysis is not available. However, for many PDE applications, such a theory is rather involved and, as the proof of \cite[Lemma 2]{HungJuan} and the unified approach developed by Lenzmann and Schikorra (and ours here) shows, one can get  in many cases  a much simpler approach.

\vspace{-0.3cm}

\subsubsection*{Notations} in $\er^{d+1}=\left \{ (x,y),\,\,x \in \er^d,\,\, y\in \er \right \}$, we will denote the classical differential operators $\nabla=(\nabla_x,\partial_y), \operatorname{div}=\operatorname{div}_x+\partial_y$ and $\Delta=\Delta_x +\partial_{yy}$ with the obvious meanings.

\section{Proof of Theorem \ref{thm-kapon1}: the case of the spectral Laplacian}

\subsection*{Harmonic extensions } (i)
We define the following three quantities which play the important role in the proof of the commutator estimate. Let $U,V,W$ be the solutions  given by Theorem \ref{ST} of the boundary value problems:
\begin{equation}\label{kato-pon-eq1-a}
\left\{
\begin{array}
{ll}%
\operatorname{div}(y^{1-2\alpha}\nabla U)=0~~&\text{in }\Omega\times(0,\infty),\\
U=0 &\text{on}
~~\partial\Omega\times (0,\infty),
\\
U(x,0)=g(x) &\text{in}
~~\Omega,
\\
\end{array}
\right.
\end{equation}
\begin{equation}\label{kato-pon-eq2-a}
\left\{
\begin{array}
{ll}%
\operatorname{div}(y^{1-2\alpha}\nabla V)=0~~&\text{in }\Omega\times(0,\infty),\\
V=0 &\text{on}
~~\partial\Omega\times (0,\infty),
\\
V(x,0)=h(x) &\text{in}
~~\Omega,
\\
\end{array}
\right.
\end{equation}
\begin{equation}\label{kato-pon-eq3-a}
\left\{
\begin{array}
{ll}%
\operatorname{div}(y^{1-2\alpha}\nabla W)=0~~&\text{in }\Omega\times(0,\infty),\\
W=0 &\text{on}
~~\partial\Omega\times (0,\infty),
\\
W(x,0)=g(x)h(x) &\text{in}
~~\Omega.
\\
\end{array}
\right.
\end{equation}
We introduce the new function $Z=W-UV$. By construction we have
\begin{align*}
-\lim\limits_{y\to 0^+} y^{1-2\alpha}Z_y(x,y)=\left[(-\Delta)_{sp}^{\alpha}(gh)(x)-g(x)(-\Delta)_{sp}^{\alpha}(h)(x)-h(x)(-\Delta)_{sp}^{\alpha}(g)(x)\right]\,.
\end{align*}
$Z$ satisfies a state equation and boundary data:
\begin{equation}\label{kato-pon-eq5-a}
\left\{
\begin{array}
{ll}%
\operatorname{div}(y^{1-2\alpha}\nabla Z)=-2y^{1-2\alpha}\nabla U \cdot\nabla V~~&\text{in }\Omega\times(0,\infty),\\
Z=0 &\text{on}
~~\partial\Omega\times (0,\infty),
\\
Z(x,0)=0 &\text{in}
~~\Omega.
\\
\end{array}
\right.
\end{equation}
Multiplying \eqref{kato-pon-eq5-a} by $\varphi=y^{-2\alpha}Z$ and integrating by parts leads
		\begin{align}\label{Z1}
	&\int_{0}^{\infty}\int_\Omega y^{1-2\alpha}Z_y\partial_y(y^{-2\alpha}Z) dxdy+\int_{\Omega} \lim\limits_{y\to 0}\left[ y^{1-2\alpha}Z_y y^{-2\alpha}Z\right] dx\\&~~+\int_{0}^{\infty}\int_\Omega y|\nabla_x (y^{-2\alpha} Z)|^2dxdy=2\int_{0}^{\infty}\int_\Omega y^{1-2\alpha}\nabla U \cdot \nabla V  y^{-2\alpha}Z dxdy.\nonumber
	\end{align}
	Since $y^{1-2\alpha}Z_y= y \partial_y (y^{-2\alpha}Z)+2\alpha y^{-2\alpha} Z$
	\begin{align}\nonumber
	\int_{0}^{\infty}\int_\Omega y^{1-2\alpha}Z_y\partial_y(y^{-2\alpha}Z) dxdy&=\int_{0}^{\infty}\int_\Omega y|\partial_y (y^{-2\alpha}Z)|^2dxdy+\alpha \int_{0}^{\infty}\int_\Omega  \partial_y\left[(y^{-2\alpha}Z)^2\right]\\&
	=\int_{0}^{\infty}\int_\Omega y|\partial_y (y^{-2\alpha}Z)|^2dxdy-\alpha \int_\Omega \lim\limits_{y\to 0}|y^{-2\alpha}Z|^2 dx.\label{Z2}
	\end{align}
	Thanks to $Z(x,0)=0$, one has
$$
	\lim\limits_{y\to 0} y^{-2\alpha}Z= \frac1{2\alpha} \nc \lim\limits_{y\to 0} y^{1-2\alpha}Z_y~~\text{a.e  in}~~\Omega.
$$
Combining this with \eqref{Z2} and  \eqref{Z1} we arrive at 	
	\begin{align}
	&\frac{1}{4\alpha}\int_{\Omega} |\lim\limits_{y\to 0} y^{1-2\alpha}Z_y(x,y)|^2 dx+\int_{0}^{\infty}\int_\Omega y|\nabla_x (y^{-2\alpha} Z)|^2dxdy\\&~~~=2\int_{0}^{\infty}\int_\Omega y^{1-2\alpha}\nabla U \cdot \nabla V  y^{-2\alpha}Z dxdy.\nonumber
	\end{align}
	
\noindent Our goal is to estimate the first term in the left-hand side in order to get the commutator estimate of Theorem \ref{thm-kapon1}. Discarding the second term  and using H\"older's inequality on the last term, we get
	\begin{align*}
	&\int_{\Omega} |\lim\limits_{y\to 0} y^{1-2\alpha}Z_y(x,y)|^2 dx\\&\leq C  ||y^{1-2\alpha}|\nabla U|||_{L^\infty(\Omega\times (0,\infty))} \int_{0}^{\infty}\int_\Omega  |\nabla V|  y^{-2\alpha}|Z| dxdy
	\\&\leq C \left(\int_{0}^{\infty}\int_\Omega  y^{1-2\beta}|\nabla V|^2 dxdy\right)^{1/2}  \left(\int_{0}^{\infty}\int_\Omega y^{-1-2(2\alpha-\beta)}| Z|^2 dxdy\right)^{1/2}\\&~~~
	 \times ||y^{1-2\alpha}|\nabla U|\,||_{L^\infty( (0,\infty)\times\Omega)}
	\end{align*}
for $0<\beta\leq \alpha<\frac{1}{2}$.   At this point, it is enough to show the following three inequalities:
	\begin{align}\label{es2}
	\int_{0}^{\infty}\int_\Omega y^{1-2\beta}|\nabla V|^2dxdy\leq C \int_\Omega |(-\Delta)_{sp}^{\frac{\beta}{2}} h|^2dx,
	\end{align}
	
	\begin{align}\label{es43}
	\int_{0}^{\infty}\int_\Omega y^{-1-2\alpha}| Z|^2 dxdy\leq C ||(-\Delta)_{sp}^{\frac{\beta}{2}}(h)||_{L^2(\Omega)}^2 ||g||_{L^\infty(\Omega)}^{\frac{2\beta}{\alpha}} ||(-\Delta)_{sp}^\alpha (g)||_{L^\infty(\Omega)}^{\frac{2(\alpha-\beta)}{\alpha}},
	\end{align}
	and finally
	\begin{align}\label{es42}||y^{1-2\alpha}|\nabla U|\,||_{L^\infty(\Omega\times (0,\infty))}\leq C||(-\Delta)_{sp}^\alpha (g)||_{L^\infty(\Omega)}.
	\end{align}
	
\medskip

\noindent \textbf{Proof of \eqref{es2}:} Let $\phi$ be the unique solution of
	\begin{equation}\label{kato-pon-eq4}
\left\{
\begin{array}
{ll}%
\operatorname{div}(y^{1-2\beta}\nabla \phi)=0~~&\text{in }\Omega\times(0,\infty),\\
\phi=0 &\text{on}
~~\partial\Omega\times (0,\infty),
\\
\phi(x,0)=h(x) &\text{in}
~~\Omega.
\\
\end{array}
\right.
\end{equation}
It is enough to check that
\begin{align}\label{es1}
\int_{0}^{\infty}\int_\Omega y^{1-2\beta}|\nabla V|^2\leq C\int_{0}^{\infty}\int_\Omega y^{1-2\beta}|\nabla\phi|^2.
\end{align}
since {$\int_\Omega |(-\Delta)_{sp}^{\frac{\beta}{2}} h|^2dx= c\int_{0}^{\infty}\int_\Omega y^{1-2\beta}|\nabla\phi|^2$ by the very definition of the extension and for a universal constant $c$ (see \cite{Stinga})} . Since 
$$\operatorname{div}(y^{1-2\alpha}\nabla \phi)=2(\beta-\alpha)y^{-2\alpha}\phi_y,$$
 the function $\overline{\phi}=\phi-V$ satisfies
	\begin{equation}\label{kato-pon-eq6}
\left\{
\begin{array}
{ll}%
\operatorname{div}(y^{1-2\alpha}\nabla \overline{\phi})=2(\beta-\alpha)y^{-2\alpha}\phi_y~~&\text{in }\Omega\times(0,\infty),\\
\overline \phi=0 &\text{on}
~~\partial\Omega\times (0,\infty),
\\
\overline{\phi}(x,0)=0 &\text{in}
~~\Omega.
\\
\end{array}
\right.
\end{equation}

Choosing $y^{2(\alpha-\beta)}\overline{\phi}$ as a test function for the above equation, one has
\begin{align*}
\int_{0}^{\infty}\int_\Omega y^{1-2\beta}|\nabla\overline{\phi}|^2+2(\alpha-\beta)y^{-2\beta}\overline{\phi}_y \overline{\phi} =2(\alpha-\beta)\int_{0}^{\infty}\int_\Omega y^{-2\beta}\phi_y \overline{\phi}.
\end{align*}
Notice that 
\begin{align*}
2(\alpha-\beta) \int_{0}^{\infty}\int_\Omega y^{-2\beta}\overline{\phi}_y \overline{\phi}dxdy =2(\alpha-\beta)\beta\int_{0}^{\infty}\int_\Omega y^{-1-2\beta}|\overline{\phi}|^2\geq 0.
\end{align*}
Furthermore, using H\"older's and Young's inequality, one gets for some $\varepsilon >0$ to be chosen 
\begin{align*}
& \int_{0}^{\infty}\int_\Omega y^{1-2\beta}|\nabla\overline{\phi}|^2+ 2(\alpha-\beta)\beta \int_{0}^{\infty}\int_\Omega y^{-1-2\beta}|\bar \phi|^2 \\&\leq  (\alpha-\beta) \Big ( \varepsilon  \int_{0}^{\infty}\int_\Omega y^{-1-2\beta}|\bar \phi|^2 + \frac{1}{\varepsilon} \int_{0}^{\infty}\int_\Omega y^{1-2\beta}|\nabla \phi|^2 \Big )  .
\end{align*}

Choosing $\varepsilon$ small enough and discarding the nonnegative term in the LHS, this implies that 
 
\begin{align*}
\int_{0}^{\infty}\int_\Omega y^{1-2\beta}|\nabla\overline{\phi}|^2  \leq C \int_{0}^{\infty}\int_\Omega y^{1-2\beta}|\nabla \phi|^2dxdy,
\end{align*}
which implies \eqref{es1} recalling that $\bar \phi=\phi-V$.

\medskip

\noindent \textbf{Proof of \eqref{es42}:} 	By standard regularity theory for the heat kernel ( see e.g. \cite{davies}):  for all $\phi\in L^\infty$ and all $x \in \Omega$
	\begin{align*}
	|\nabla_x e^{t\Delta_{sp}}\phi(x)|\leq C\frac{1}{\sqrt{t}}\min\left\{1,\frac{1}{t^{\frac{d}{2}}}\right\}||\phi||_{L^\infty(\Omega)},~
	|e^{t\Delta_{sp}}\phi(x)|\leq C\min\left\{1,\frac{1}{t^{\frac{d}{2}}}\right\}||\phi||_{L^\infty(\Omega)}.
	\end{align*}

	Using  these estimates in  formulas  \eqref{formulaBVP} (and obvious changes of variables), one gets 
	\begin{align*}
		&	|y|^{1-2\alpha}|\nabla U(x,y)|\leq C|y|^{1-2\alpha}\int_{0}^{\infty}\left(\frac{1}{\sqrt{t}}+\frac{|y|}{t}\right)e^{-\frac{y^2}{4t}}\frac{dt}{t^{1-\alpha}}||(-\Delta)_{sp}^\alpha g||_{L^\infty(\Omega)},\\&
				|y||\nabla U(x,y)|\leq C|y|^{1+2\alpha}\int_{0}^{\infty}\left(\frac{1}{|y|}+\frac{|y|}{t}+\frac{1}{\sqrt{t}}\right)e^{-\frac{y^2}{4t}}
		\frac{dt}{t^{1+\alpha}}||g||_{L^\infty(\Omega)},
	\end{align*}
	which implies \eqref{es42} and 
	\begin{align}\label{es39}
		||y|\nabla U|\,||_{L^\infty(\Omega\times (0,\infty))}\leq C||g||_{L^\infty(\Omega)}.
	\end{align}

\medskip

\noindent 	\textbf{Proof of \eqref{es43}:} We choose  $\varphi=y^{-2(\alpha-\beta)}Z$ as a test function for  \eqref{kato-pon-eq5-a}
	\begin{align}\label{es3}
	&\int_{0}^{\infty}\int_\Omega y^{1-2(2\alpha-\beta)}|\nabla Z|^2dxdy-2(\alpha-\beta)	\int_{0}^{\infty}\int_\Omega y^{-2(2\alpha-\beta)} Z_yZ dxdy\\&=2\int_{0}^{\infty}\int_\Omega y^{1-2(2\alpha-\beta)}\nabla  U \cdot \nabla V Z dxdy.
	\end{align}
	Using the following Hardy's inequality (see \cite{zygmund}, p.20)
	\begin{align}\label{es4}
	\int_{0}^{\infty} y^{1-2(2\alpha-\beta)}|Z_y|^2dy\geq (2\alpha-\beta)^2 \int_{0}^{\infty}y^{-1-2(2\alpha-\beta)}|Z|^2dy,
	\end{align}and the fact that
	\begin{align*}
	-2(\alpha-\beta) 	\int_{0}^{\infty}\int_\Omega y^{-2(2\alpha-\beta)} Z_yZ dxdy=-2(\alpha-\beta)(2\alpha-\beta) \int_{0}^{\infty}\int_\Omega y^{-1-2(2\alpha-\beta)}|Z|^2dxdy,
	\end{align*} yields
	\begin{align*}
	&\frac{\beta}{2\alpha-\beta}\int_{0}^{\infty}\int_\Omega y^{1-2(2\alpha-\beta)}|\nabla Z|^2dxdy\\&~~\leq LHS\eqref{es3}= 2\int_{0}^{\infty}\int_\Omega y^{1-2(2\alpha-\beta)}\nabla U  \cdot \nabla V Z dxdy.
	\end{align*}
	Using H\"older's inequality (together with Young's inequality) and \eqref{es4}, we get
	\begin{align*}
	&\int_{0}^{\infty}\int_\Omega y^{1-2(2\alpha-\beta)}|\nabla Z|^2 \leq C\int_{0}^{\infty}\int_\Omega y^{3-2(2\alpha-\beta)} |\nabla V|^2 |\nabla U|^2
	\\&\leq C ||y^{1-2(\alpha-\beta)}|\nabla U|\,||_{L^\infty(\Omega\times(0,\infty))}^2 \int_{0}^{\infty}\int_\Omega y^{1-2\beta} |\nabla V|^2
	\\&\overset{\eqref{es2}}\leq C ||(-\Delta)^{\frac{\beta}{2}}(h)||_{L^2(\Omega)}^2 ||y|\nabla U|||_{L^\infty(\Omega\times(0,\infty))}^{\frac{2\beta}{\alpha}}||y^{1-2\alpha}|\nabla U|||_{L^\infty(\Omega\times(0,\infty))}^{\frac{2(\alpha-\beta)}{\alpha}}\\&\overset{\eqref{es42},\eqref{es39}}\leq C ||(-\Delta)^{\frac{\beta}{2}}(h)||_{L^2(\Omega)}^2 ||g||_{L^\infty(\Omega)}^{\frac{2\beta}{\alpha}} ||(-\Delta)^\alpha (g)||_{L^\infty(\Omega)}^{\frac{2(\alpha-\beta)}{\alpha}},
	\end{align*}
	which implies \eqref{es43}. The proof is complete. \qed

	
	\section{Proof of Theorem \ref{thm-kapon2}}

	The case of the Fractional Laplacian  in $\mathbb R^d$  is a straightforward modification of the previous arguments. We sketch it here for the reader's convenience. First we introduce the appropriate extensions, modified to suit our operator.
	 Let $U,V,W$ be the solutions  given by Theorem \ref{CaffaS} of the boundary value problems:
\begin{equation}\label{kato-pon-eq1-a2}
\left\{
\begin{array}
{ll}%
\operatorname{div}(y^{1-2\alpha}\nabla U)=0~~&\text{in }\mathbb R^d \times(0,\infty),\\
U(x,0)=g(x) &\text{in}
~~\mathbb R^d,
\\
\end{array}
\right.
\end{equation}
\begin{equation}\label{kato-pon-eq2-a2}
\left\{
\begin{array}
{ll}%
\operatorname{div}(y^{1-2\alpha}\nabla V)=0~~&\text{in }\mathbb R^d\times(0,\infty),\\
V(x,0)=h(x) &\text{in}
~~\mathbb R^d,\end{array}
\right.
\end{equation}
\begin{equation}\label{kato-pon-eq3-a2}
\left\{
\begin{array}
{ll}%
\operatorname{div}(y^{1-2\alpha}\nabla W)=0~~&\text{in }\mathbb R^d \times(0,\infty),\\
W(x,0)=g(x)h(x) &\text{in}
~~\mathbb R^d,
\end{array}
\right.
\end{equation}
Defining as before  $Z=W-UV$, we have by construction
\begin{align*}
-\lim\limits_{y\to 0^+} y^{1-2\alpha}Z_y(x,y)=\left[(-\Delta)^{\alpha}(gh)(x)-g(x)(-\Delta)^{\alpha}(h)(x)-h(x)(-\Delta)^{\alpha}(g)(x)\right]
\end{align*}
and $Z$ solves the extended problem
\begin{equation}\label{kato-pon-eq5-a2}
\left\{
\begin{array}
{ll}%
\operatorname{div}(y^{1-2\alpha}\nabla Z))=-2y^{1-2\alpha}\nabla U \cdot\nabla V~~&\text{in }\mathbb R^d\times(0,\infty),\\
Z=0 &\text{on}
~~\mathbb R^d,
\end{array}
\right.
\end{equation}
Multiplying \eqref{kato-pon-eq5-a2} by  check this  $\varphi=y^{-2\alpha}Z$  and integrating by parts leads

	\begin{align*}
	&\frac{1}{4\alpha}\int_{\mathbb R^d} |\lim\limits_{y\to 0} y^{1-2\alpha}Z_y(x,y)|^2 dx+\int_{0}^{\infty}\int_{\mathbb R^d} y|\nabla (y^{-2\alpha} Z)|^2dxdy\\&~~~=2\int_{0}^{\infty}\int_{\mathbb R^d} y^{1-2\alpha}\nabla U \cdot \nabla V  y^{-2\alpha}Z dxdy.
	\end{align*}
	 Using H\"older's inequality as before,
	\begin{align*}
	&\int_{\mathbb R^d} |\lim\limits_{y\to 0} y^{1-2\alpha}Z_y(x,y)|^2 dx\\&\leq C  ||y^{1-2\alpha}|\nabla U|||_{L^\infty(\mathbb R^d \times (0,\infty))} \int_{0}^{\infty}\int_{\mathbb R^d}  |\nabla V|  y^{-2\alpha}|Z| dxdy
	\\&\leq C \left(\int_{0}^{\infty}\int_{\mathbb R^d}  y^{1-2\beta}|\nabla V|^2 dxdy\right)^{1/2}  \left(\int_{0}^{\infty}\int_{\mathbb R^d} y^{-1-2(2\alpha-\beta)}| Z|^2 dxdy\right)^{1/2}\\&~~~
	 \times ||y^{1-2\alpha}|\nabla U|\,||_{L^\infty( (0,\infty)\times\mathbb R^d)}.
	\end{align*}
	We just need to show then
	\begin{align}\label{es22}
	\int_{0}^{\infty}\int_{\mathbb R^d} y^{1-2\beta}|\nabla V|^2dxdy\leq C \int_{\mathbb R^d} |(-\Delta)^{\frac{\beta}{2}} h|^2dx,
	\end{align}
	
	\begin{align}\label{es432}
	\int_{0}^{\infty}\int_{\mathbb R^d} y^{-1-2\alpha}| Z|^2 dxdy\leq C ||(-\Delta)^{\frac{\beta}{2}}(h)||_{L^2(\mathbb R^d)}^2 ||g||_{L^\infty(\mathbb R^d)}^{\frac{2\beta}{\alpha}} ||(-\Delta)^\alpha (g)||_{L^\infty(\mathbb R^d)}^{\frac{2(\alpha-\beta)}{\alpha}},
	\end{align}
	and
	\begin{align}\label{es422}||y^{1-2\alpha}|\nabla U|\,||_{L^\infty(\mathbb R^d\times (0,\infty))}\leq C||(-\Delta)^\alpha (g)||_{L^\infty(\mathbb R^d)}.
	\end{align}

Compared to the previous section, the only point which is different is the one to prove \eqref{es422}. Here we just invoke the explicit expression of the Poisson kernel $P_\alpha$ which is (up to a universal normalizing constant)
\begin{equation*}
P_\alpha(x,y)=\frac{y^{2\alpha}}{(|x|^2+y^2)^{\frac{d+2\alpha}{2}}}
\end{equation*}
and by construction,  the solution $U$ which is the convolution of $P_\alpha$ with $g$ satisfies the desired estimates.

\section*{Appendix}

In this appendix, we provide two results supporting our conjecture on the failure of the usual form of the Leibniz rule in the case of the restricted Laplacian. Our purpose is to relate a weighted (by a suitable power of the distance function) $L^p$ norm for $p=1,2$ of the function to the $L^p$ norm of its fractional Laplacian. We would like to make in particular three comments:

\begin{itemize}
\item by the very definition of the restricted Laplacian, since the functions are supported on $\Omega$, the Leibniz rule reduces to estimate the integrals
$$
 \int_{\er^d} \frac{(f(x)-f(y))\,( g(x)-g(y) )}{|x-y|^{d+2\alpha}}\,dy
$$
and since we are interested in estimating $L^2$ norms in $\Omega$, one is led to consider quantities of the type
$$
\int_{\Omega \times \Omega} \frac{(f(x)-f(y))\,( g(x)-g(y) )}{|x-y|^{d+2\alpha}}\,dx\,dy,\,\,\,\,\,\int_{\Omega \times \Omega^c} \frac{(f(x)-f(y))\,( g(x)-g(y) )}{|x-y|^{d+2\alpha}}\,dx\,dy
$$

\item it is by now well-known that smooth functions that are compactly supported in $\Omega$ and have finite $H^\alpha$ semi-norm, behave like $\dist(x,\partial \Omega)^{\alpha}$  close to the boundary of $\Omega$.

\item finally, notice that there is two different ways to define a semi-norm (even in $\er^d$) in $\dot W^{\alpha,p}(\er^d)$, namely
$$
\int_{\er^d \times \er^d} \frac{|f(x)-f(y)|^p}{|x-y|^{d+p\alpha}}\,dx\,dy\,\,\,\,\,\text{and}\,\,\,\,\,\int_{\er^d} |(-\Delta)^{\alpha/2} f(x)|^p\,dx.
$$
In the case of the whole space $\er^d$ and $p=2$, these latter norms are equivalent. Actually, according to \cite{stein}, depending on $p$, these spaces are ordered for every $\alpha \in (0,1)$ in bounded domains and they are still equivalent for $p=2$.
\end{itemize}

According to the previous remarks, if one seeks for a counter-example, one would need to understand how the $L^2$ norm in $\er^d$ of the commutator behaves with respect to its $L^2$ norm in $\Omega$. The following computations show that the boundary behaviour plays a crucial role.

Let $\alpha\in (0,1/2)$. Let $u_\varepsilon\in C_c^\infty(B_1(0))$ be a cut-off function such that $u_\varepsilon=1$ in $B_{1-2\varepsilon}$ , $u_\varepsilon=0$ in $B_{1-\varepsilon}^c$ and $|\nabla u_\varepsilon|\leq C\varepsilon$, $0\leq u_\varepsilon\leq 1$. We easily get  first

\begin{equation}
\int_{B_1}\frac{u_\varepsilon(x)^2}{(1-|x|^2)^{2\alpha}} dx\sim 1, ~~\forall \varepsilon\in (0,1/10).
\end{equation}
and for any $0<\alpha<\alpha_0<1/2$
\begin{equation}\label{appZ1}
\int_{B_1}\int_{B_1} \frac{|u_\varepsilon(x)-u_\varepsilon(y)|^2}{|x-y|^{d+2\alpha}} dxdy\lesssim_{s_0} \varepsilon^{1-2\alpha_0}.
\end{equation}
Indeed,
\begin{align*}
&\int_{B_1}\int_{B_1} \frac{|u_\varepsilon(x)-u_\varepsilon(y)|^2}{|x-y|^{d+2\alpha}} dxdy\\&=2\int_{B_1\backslash B_{1-3\varepsilon}}\int_{B_1\backslash B_{1-3\varepsilon}}\frac{|u_\varepsilon(x)-u_\varepsilon(y)|^2}{|x-y|^{d+2\alpha}} dxdy+2\int_{B_1\backslash B_{1-2\varepsilon}}\int_{ B_{1-3\varepsilon}}\frac{|1-u_\varepsilon(y)|^2}{|x-y|^{d+2\alpha}} dxdy\\&\lesssim \varepsilon^{-2\alpha_0}
\int_{B_1\backslash B_{1-3\varepsilon}}\int_{B_1\backslash B_{1-3\varepsilon}}\frac{1}{|x-y|^{d-2(\alpha_0-\alpha)}} dxdy+\int_{B_1\backslash B_{1-2\varepsilon}}\int_{ B_{1-3\varepsilon}}\frac{1}{|x-y|^{d+2\alpha}} dxdy
\\&\lesssim \varepsilon^{-2\alpha_0}
\int_{B_1\backslash B_{1-3\varepsilon}} dy+\int_{B_1\backslash B_{1-2\varepsilon}}\varepsilon^{-2\alpha}dy\\&\lesssim \varepsilon^{1-2\alpha_0}+\varepsilon^{1-2\alpha}\lesssim\varepsilon^{1-2\alpha_0}.
\end{align*}
On the other hand, for any $x\in B_1$
\begin{equation}\label{appZ3}
||(-\Delta)^{\alpha/2}u_\varepsilon||_{L^2(B_1)}\sim 1.
\end{equation}
Indeed, for $x\in B_1$
\begin{align*}
|(-\Delta)^{\alpha/2}u_\varepsilon(x)-\int_{B_1}\frac{u_\varepsilon(x)-u_\varepsilon(y)}{|x-y|^{d+\alpha}} dy|&=|\int_{B_1^c}\frac{u_\varepsilon(x)}{|x-y|^{d+\alpha}} dy|\\&\lesssim \frac{|u_\varepsilon(x)|}{(1-|x|^2)^\alpha}.
\end{align*}
So,
\begin{align}\nonumber
\int_{B_1}|(-\Delta)^{\alpha/2}u_\varepsilon(x)-\int_{B_1}\frac{u_\varepsilon(x)-u_\varepsilon(y)}{|x-y|^{d+\alpha}} dy|^2dx&=|\int_{B_1^c}\frac{u_\varepsilon(x)}{|x-y|^{d+\alpha}} dy|\\&\lesssim \int_{B_1}\frac{|u_\varepsilon(x)|^2}{(1-|x|^2)^{2\alpha}} dx\sim 1.\label{appZ2}
\end{align}
Moreover,  for $\alpha<\alpha_1<\alpha_2<1/2$
\begin{align*}
&\int_{B_1}|\int_{B_1}\frac{u_\varepsilon(x)-u_\varepsilon(y)}{|x-y|^{d+\alpha}} dy|^2dx
\\& \lesssim_{\alpha_1} \int_{B_1}\int_{B_1}\frac{|u_\varepsilon(x)-u_\varepsilon(y)|^2}{|x-y|^{d+2\alpha_1}} dydx
\\&\overset{\eqref{Z1}}\lesssim_{\alpha_1,\alpha_2} \varepsilon^{1-2\alpha_2}.
\end{align*}
Combining this with \eqref{appZ2} yields \eqref{appZ3}. As a consequence, we have
\begin{equation*}
\int_{B_1}\frac{u_\varepsilon(x)^2}{(1-|x|^2)^{2\alpha}} dx\sim ||(-\Delta)^{\alpha/2}u_\varepsilon||_{L^2(B_1)}^2\sim 1.
\end{equation*}
but
\begin{equation*}
\int_{B_1}\int_{B_1} \frac{|u_\varepsilon(x)-u_\varepsilon(y)|^2}{|x-y|^{d+2\alpha}} dxdy\to 0 ~~\text{as}~~ \varepsilon\to 0.
\end{equation*}

This is an explicit example that in a bounded domain for $\alpha <1/2$, the Hardy inequality in $L^2$ does not hold.

\vspace{0.5cm}
However, the analogous result in $L^1$ does hold:
\begin{theorem}Let $u\in C^\infty_c(\Omega)$ be a solution to
	\begin{equation}
\int_{\mathbb{R}^d}\frac{u(x)-u(y)}{|x-y|^{d+2\alpha}} dy =f(x)
	\end{equation}
	in $\Omega$, where $\Omega$ is smooth bounded domain.  Then,
	\begin{equation}\label{Z5}
	\int_{\Omega}\frac{|u(x)|}{dist(x,\partial\Omega)^{2\alpha}} dx+||u||_{W^{2\alpha-\delta,1}(\mathbb{R}^d)}\lesssim_{\Omega,\delta} ||f||_{L^1(\Omega)}
	\end{equation}
	for any $\delta\in (0,\alpha/4)$
\end{theorem}
\begin{proof}
	Using $T_\varepsilon(u(x))=\sign (u(x))\min\{\varepsilon,|u(x)|\}$ as test function,
	\begin{equation}
	\int_{\mathbb{R}^d}\int_{\mathbb{R}^d}\frac{|u(x)-u(y)| |T_\varepsilon(u(x))-T_\varepsilon(u(y))|}{|x-y|^{d+2\alpha}} dxdy=2\int T_\varepsilon(u(x))f(x) dx.
	\end{equation}
	This implies
	\begin{align*}
	\int_{\Omega^c}\int_{\Omega}\frac{|u(x)| |T_\varepsilon(u(x))|}{|x-y|^{d+2\alpha}} dxdy=\int_{\Omega^c}\int_{\Omega}\frac{|u(x)-u(y)| |T_\varepsilon(u(x))-T_\varepsilon(u(y))|}{|x-y|^{d+2\alpha}} dxdy\leq \varepsilon |f||_{L^1(\Omega)}.
	\end{align*}
	So,
	\begin{align*}
	\int_{\Omega^c}\int_{\Omega}\frac{|u(x)| \varepsilon^{-1} T_\varepsilon(|u(x)|)}{|x-y|^{d+2\alpha}} dxdy\leq  |f||_{L^1(\Omega)}.
	\end{align*}
	Letting $\varepsilon\to 0$ to get
$$
	\int_{\Omega^c}\int_{\Omega}\frac{|u(x)|}{|x-y|^{d+2\alpha}} dxdy\leq  ||f||_{L^1(\Omega)}.$$
	Since
$$
	\int_{\Omega^c}\frac{1}{|x-y|^{d+2\alpha}} dy\sim \dist(x,\partial\Omega)^{-2\alpha},$$
	so,
	\begin{equation}\label{Z4}
	\int_{\Omega}\frac{|u(x)|}{dist(x,\partial\Omega)^{2\alpha}} dx\lesssim ||f||_{L^1(\Omega)}.
	\end{equation}
	Moreover,
	we can write
$$
	\int_{\mathbb{R}^d}\frac{u(x)-u(y)}{|x-y|^{d+2\alpha}} dy =\mathbf{1}_{x\in \Omega}f(x)+\mathbf{1}_{x\notin \Omega}\int_{\Omega}\frac{-u(y)}{|x-y|^{d+2\alpha}} dy:=g(x),$$
	for any $x\in \mathbb{R}^d$.
	Thus, by the standard regularity theory, one has \begin{equation}
	||u||_{W^{2\alpha-\delta,1}(\mathbb{R}^d)}\lesssim_{\Omega,\delta} ||g||_{L^1(\mathbb{R}^d)},
	\end{equation}
	for any $\delta\in (0,\alpha/2)$.
	Since
	\begin{align*}
	||g||_{L^1(\mathbb{R}^d)}&\leq ||f||_{L^1(\Omega)}+\int_{\Omega^c}\int_{\Omega}\frac{|u(y)|}{|x-y|^{d+2\alpha}} dydx\\&\lesssim  ||f||_{L^1(\Omega)}+\int_{\Omega}\frac{|u(y)|}{\dist(y,\partial\Omega)^{2\alpha}} dy\\&\overset{\eqref{Z4}}\lesssim ||f||_{L^1(\Omega)},
	\end{align*}
	we get \eqref{Z5}.
\end{proof}

\vspace{1cm}

\noindent {\textbf{\large \sc Acknowledgments.}} Quoc-Hung Nguyen is  supported  by the Shanghai Tech University startup fund. J.L.V. partially funded by Project  PGC2018-098440-B-I00 from  MICINN,  of the Spanish Government. Partially performed as an Honorary Professor at Univ. Complutense de Madrid.

\bibliographystyle{alpha}
\bibliography{biblio}

\vskip 1cm

{2010 Mathematics Subject Classification.} {Primary: 42B37.
Secondary:  35J25}

{\bf Keywords: } Fractional Laplacian operators on domains, commutator estimates, Leibniz rule .

\

\end{document}